\documentclass{amsart}
\usepackage{amssymb}
\usepackage{graphicx}

\newtheorem{thm}{Theorem}
\newtheorem{cor}{Corollary}
\newtheorem{lem}{Lemma}

\newtheorem{rem}{Remark}
\newtheorem{conj}{Conjecture}

\newcommand{\es}{{\mathcal S}}

\newcommand{\D}{{\mathbb D}}

\def\be{\begin{equation}}
\def\ee{\end{equation}}

\newcommand{\bee}{\begin{enumerate}}
\newcommand{\eee}{\end{enumerate}}

\newcommand{\blem}{\begin{lem}}
\newcommand{\elem}{\end{lem}}
\newcommand{\bthm}{\begin{thm}}
\newcommand{\ethm}{\end{thm}}
\newcommand{\bcor}{\begin{cor}}
\newcommand{\ecor}{\end{cor}}
\newcommand{\beg}{\begin{example}}
\newcommand{\eeg}{\end{example}}
\newcommand{\begs}{\begin{examples}}
\newcommand{\eegs}{\end{examples}}
\newcommand{\bdefe}{\begin{defin}}
\newcommand{\edefe}{\end{defin}}
\newcommand{\bprob}{\begin{prob}}
\newcommand{\eprob}{\end{prob}}
\newcommand{\bei}{\begin{itemize}}
\newcommand{\eei}{\end{itemize}}

\newcommand{\bcon}{\begin{conj}}
\newcommand{\econ}{\end{conj}}
\newcommand{\bcons}{\begin{conjs}}
\newcommand{\econs}{\end{conjs}}
\newcommand{\bprop}{\begin{propo}}
\newcommand{\eprop}{\end{propo}}
\newcommand{\br}{\begin{rem}}
\newcommand{\er}{\end{rem}}
\newcommand{\brs}{\begin{rems}}
\newcommand{\ers}{\end{rems}}
\newcommand{\bo}{\begin{obser}}
\newcommand{\eo}{\end{obser}}
\newcommand{\bos}{\begin{obsers}}
\newcommand{\eos}{\end{obsers}}
\newcommand{\bpf}{\begin{pf}}
\newcommand{\epf}{\end{pf}}
\newcommand{\ba}{\begin{array}}
\newcommand{\ea}{\end{array}}
\newcommand{\beq}{\begin{eqnarray}}
\newcommand{\beqq}{\begin{eqnarray*}}
\newcommand{\eeq}{\end{eqnarray}}
\newcommand{\eeqq}{\end{eqnarray*}}

\begin{document}
\bibliographystyle{amsplain}

\title[Some properties of bi-univalent functions]{On some properties of bi-univalent functions in the unit disc}

\author[M. Obradovi\'{c}]{Milutin Obradovi\'{c}}
\address{Department of Mathematics,
Faculty of Civil Engineering, University of Belgrade,
Bulevar Kralja Aleksandra 73, 11000, Belgrade, Serbia}
\email{obrad@grf.bg.ac.rs}

\author[N. Tuneski]{Nikola Tuneski}
\address{Department of Mathematics and Informatics, Faculty of Mechanical Engineering, Ss. Cyril and Methodius
University in Skopje, Karpo\v{s} II b.b., 1000 Skopje, Republic of North Macedonia.}
\email{nikola.tuneski@mf.edu.mk}

\author[P. Zaprawa]{Pawe{\l} Zaprawa}
\address{Department of Mathematic, Faculty of Mechanical Engineering, Lublin University of Technology, Poland.}
\email{p.zaprawa@pollub.pl}



\begin{abstract}
In this paper we use a method based on the Grunsky coefficients to find upper bounds of the modulus of the initial coefficients, difference of the moduli of two consecutive initial coefficients, of the modulus of the initial logarithmic coefficient, and of the second Hankel determinant for the class of normalized bi-univalent functions.
\end{abstract}

\maketitle

\section{Introduction and definitions}

\medskip

Let $\mathcal{A}$ be the class of functions $f$ which are analytic  in the open unit disk $\D=\{z:|z|<1\}$ of the form
\be\label{e1}
f(z)=z+a_2z^2+a_3z^3+\cdots,
\ee
and let $\mathcal{S}$ be the subclass of $\mathcal{A}$ consisting of functions that are univalent in $\D$.

\medskip

Further, with $ \mathcal{S}_{b}$ we denote the class of functions $f\in \mathcal{S}$ such that its inverse function $f^{-1}$ is also univalent in $\D$.
For example the functions $z$, $\frac{z}{1-z}$, $\frac{1}{2}\ln\frac{1+z}{1-z}$ are characteristic examples of bi-univalent functions, but the Koebe function $k(z)=\frac{z}{(1-z)^{2}}$ is not bi-univalent since $k(\D)$ does not contain the unit disc. It means that $\mathcal{S}_{b}\subset \mathcal{S}$.

\medskip

Although the famous Bieberbach conjecture $|a_n|\le n$ for $n\ge2$ was proved by de Branges  in 1985 \cite{Bra85}, a great many other problems concerning the coefficients of univalent functions remain open and are studied over entire class $\es$, or, over its subclasses. Also, the problem of initial coefficients for the class $ \mathcal{S}_{b}$ is not so simple. Namely, Lewin \cite{Lewin} showed that the second coefficient of every $f\in \mathcal{S}_{b}$ satisfy the inequality $|a_{2}|\leq 1.51$. Later, D.L. Tan \cite{Tan} obtained better result $|a_{2}|\leq 1.485$ and this estimate is the best known. Other results for the same coefficient we can find in \cite{Ked-1985, KKR, Net-1969}.

\medskip

In this paper we give estimates for the initial coefficients $a_{3},\,a_{4},\,a_{5}$ of functions belonging to the class
$ \mathcal{S}_{b}$. Also, some problems concerning coefficient difference,  logarithmic coefficients and  Hankel determinants in the class $ \mathcal{S}_{b}$ will be considered.

\medskip

Our main tool will be the Grunsky coefficients and its properties. Here are basic definitions and results  on those coefficients based on the book of N. A. Lebedev (\cite{Lebedev}).

\medskip

Let $f \in \mathcal{S}$ and let
\[
\log\frac{f(t)-f(z)}{t-z}=\sum_{p,q=0}^{\infty}\omega_{p,q}t^{p}z^{q},
\]
where $\omega_{p,q}$ are called Grunsky's coefficients with property $\omega_{p,q}=\omega_{q,p}$.
For those coefficients we have the next Grunsky's inequalities (\cite{duren,Lebedev}):
\be\label{eq 4}
\sum_{q=1}^{\infty}q \left|\sum_{p=1}^{\infty}\omega_{p,q}x_{p}\right|^{2}\leq \sum_{p=1}^{\infty}\frac{|x_{p}|^{2}}{p}
\ee
and
\be\label{eq 5}
\left|\sum_{p=1}^{\infty} \sum_{q=1}^{\infty} \omega_{p,q} x_{p} x_{q} \right|  \leq \sum_{p=1}^{\infty}\frac{|x_{p}|^{2}}{p},
\ee
where $x_{p}$ are arbitrary complex numbers such that $0< \sum_{p=1}^{\infty}\frac{|x_{p}|^{2}}{p}< +\infty$. If $\overline{\lim}_{p\to\infty} \sqrt[p]{|x_p|}<1$, then in \eqref{eq 4} we have equality if and only if the area of $\widehat{\mathbb{C}} \setminus g(\D)$ is zero, where $g(z)= \frac{1}{f(z)}$. In \eqref{eq 5} equality appears if and only if $\sum_{p=1}^{\infty}\omega_{p,q}x_{p} = \lambda \frac{1}{q} \overline{x}_q$, $\lambda=1,2,\ldots$, and $\lambda$ is some constant with $|\lambda|=1$.

\medskip

Further, it is well-known that if $f$ given by \eqref{e1}
belongs to $\mathcal{S}$, then also
\be\label{eq 6}
f_{*}(z)=\sqrt{f(z^{2})}=z +c_{3}z^3+c_{5}z^{5}+\cdots
\ee
belongs to the class $\mathcal{S}$. Then for the function $f_{*}$ we have the appropriate Grunsky's
coefficients of the form $\omega_{2p-1,2q-1}$ and the inequalities \eqref{eq 4} and \eqref{eq 5} receive the forms:
\be\label{eq 7}
\sum_{q=1}^{\infty}(2q-1) \left|\sum_{p=1}^{\infty}\omega_{2p-1,2q-1}x_{2p-1}\right|^{2}\leq \sum_{p=1}^{\infty}\frac{|x_{2p-1}|^{2}}{2p-1}
\ee
and
\be\label{eq-8}
\left|\sum_{p=1}^{\infty} \sum_{q=1}^{\infty} \omega_{2p-1,2q-1} x_{2p-1} x_{2q-1} \right|  \leq \sum_{p=1}^{\infty}\frac{|x_{2p-1}|^{2}}{2p-1},
\ee
respectively.

\medskip

Here, and further in the paper we omit the upper index (2) in  $\omega_{2p-1,2q-1}^{(2)}$ if compared with Lebedev's notation.

\medskip

From inequalities \eqref{eq 7} and \eqref{eq-8}, when $x_{2p-1}=0$, $p=3,4,\ldots$, we have
\[|\omega_{11} x_1 +\omega_{31} x_3 |^2 +3|\omega_{13} x_1 +\omega_{33} x_3 |^2 + 5|\omega_{15} x_1 +\omega_{35} x_3 |^2 \le |x_1|^2+\frac{|x_3|^2}{3}\]
and
\[ |\omega_{11} x_1^2 +2\omega_{13}x_1 x_3 +\omega_{33}x_3^2 |\le |x_1|^2+\frac{|x_3|^2}{3}, \]
respectively. From \eqref{eq 7}, for $x_1=1$ and $x_3=0$, we obtain
\begin{equation}\label{eq-11}
  |\omega_{11}|^2 + 3 |\omega_{13}|^2 + 5|\omega_{15}|^2 \le1,
\end{equation}
and for $x_1=0$ and $x_3=1$, we obtain
\[  |\omega_{13}|^2 + 3 |\omega_{33}|^2 + 5|\omega_{35}|^2 \le \frac13. \]

\medskip

For the functions in $\mathcal{S}$ we have the following result that we will use later on.

\medskip

\begin{lem}\label{lem-1}
Grunsky coefficient $w_{13}$  for a function $f\in\mathcal{S}$ satisfies
\[|w_{13}|\leq
\begin{cases}
\tfrac12(1+|w_{11}|^2)\ &, |w_{11}|\in[0,b]\\
\sqrt{\tfrac13(1-|w_{11}|^2)}\ &, |w_{11}|\in[b,1]\ ,
\end{cases}\]
where $b=\sqrt{\frac{2 \sqrt{7}-5}{3}} = 0.311\ldots$.
\end{lem}

The first inequality is a consequence of the inequality $|2w_{13}-w_{11}^2|\leq 1$, obtained from $|a_3-a_2^2|\leq 1$, valid for $f\in\mathcal{S}$. The second one follows from the inequality \eqref{eq-11}.

\medskip

As it has been shown in \cite[p. 57]{Lebedev}, if $f$ is given by \eqref{e1} then the coefficients $a_{2}$, $ a_{3}$, $ a_{4}$ and $a_5$ are expressed by Grunsky's coefficients  $\omega_{2p-1,2q-1}$ of the function $f_{*}$ given by
\eqref{eq 6} in the following way:
\begin{equation}\label{eq-13}
\begin{split}
a_{2}&=2\omega _{11},\\
a_{3}&=2\omega_{13}+3\omega_{11}^{2}, \\
a_{4}&=2\omega_{33}+8\omega_{11}\omega_{13}+\frac{10}{3}\omega_{11}^{3},\\
a_{5}&=2\omega_{35}+8\omega_{11}\omega_{33}+5\omega_{13}^{2}+18\omega_{11}^{2}\omega_{13}+\frac{7}{3}\omega_{11}^{4},\\
0&=3\omega_{15}-3\omega_{11}\omega_{13}+\omega_{11}^{3}-3\omega_{33},\\
0&=\omega_{17}-\omega_{35}-\omega_{11}\omega_{33}-\omega_{13}^{2}+\frac{1}{3}\omega_{11}^{4}.
\end{split}
\end{equation}

\medskip

A comprehensive overview of the application of Grunsky coefficients in the general class of univalent functions is given in \cite{2026}.

\medskip

\section{The initial coefficients in the class $ \mathcal{S}_{b}$ }

The first theorem will give estimates of the modulus of the third, the fourth and the fifth coefficient of the expansion of a normalized bi-univalent function.

\begin{thm}\label{th1}
Let $f\in\mathcal{S}_{b}$ be given by \eqref{e1}. Then
\begin{itemize}
  \item[($i$)] $|a_{3}|\leq 2.427\ldots$\ ,
  \item[($ii$)] $|a_{4}| \leq 3.461\ldots$\ ,
  \item [($iii$)]$|a_{5}|\leq 4.993\ldots$\ .
\end{itemize}
\end{thm}

\begin{proof}
Since $f$ is bi-univalent in the unit disk, we have $|a_{2}|\leq 1.485$, and further from \eqref{eq-13}, we have $|a_{2}|=2|\omega _{11}|$, i.e., $2|\omega _{11}|\leq1.485$ and
\begin{equation}\label{def_a}
|\omega _{11}|\leq 0.7425=:a\ .
\end{equation}

\medskip

\begin{itemize}
\item[(i)] From \eqref{eq-11} we have $  |\omega_{11}|^2 + 3 |\omega_{13}|^2 \le1,$ i.e., $|\omega _{13}|\leq\frac{1}{\sqrt{3}}\sqrt{1-|\omega _{11}|^{2}}$. Using this and the expression for $a_3$ from \eqref{eq-13}, we receive
\begin{equation*}
\begin{split}
|a_{3}| &=|2\omega_{13}+3\omega_{11}^{2}|\leq2|\omega_{13}|+3|\omega _{11}|^{2}\\
& \leq \frac{2}{\sqrt{3}}\sqrt{1-|\omega _{11}|^{2}}+3|\omega _{11}|^{2}=:f_{1}(|\omega _{11}|),
\end{split}
\end{equation*}
where $f_{1}(x)=\frac{2}{\sqrt{3}}\sqrt{1-x^{2}}+3x^{2}$, $x\in[0, a]$.
It is easy to verify that $f_1$ is strictly increasing on $[0,a]$, so
\[|a_3| \le f_{1}(a)\leq \frac{2}{\sqrt{3}}\sqrt{1-a^{2}}+3a^{2}=2.427\ldots .\]

\medskip

\item[(ii)] The fifth relation in \eqref{eq-13} gives
$$\omega_{33}=\omega_{15}-\omega_{11}\omega_{13}+\frac{1}{3}\omega_{11}^{3},$$
which, together with the relation for $a_4$, lead to
\begin{equation}\label{for_a4}
a_{4}=2\omega_{15}+6\omega_{11}\omega_{13}+4\omega_{11}^{3}\ .
\end{equation}
Hence,
\[|a_{4}|\leq 2|\omega_{15}|+6|\omega_{11}||\omega_{13}|+4|\omega_{11}|^{3}\ .\]
Finally, \eqref{eq-11} implies
\[ |a_{4}|\leq 4|\omega_{11}|^{3} + 6|\omega_{11}||\omega_{13}| +\frac{2}{\sqrt{5}}\sqrt{1-|\omega_{11}|^{2}-3|\omega_{13}|^{2}}=:f_{2}(|\omega_{11}|,|\omega_{13}|), \]
where
$$f_{2}(x,y)= 4x^{3}+6xy+\frac{2}{\sqrt{5}}\sqrt{1-x^{2}-3y^{2}}\ ,\ (x,y)\in\Omega$$
where, by Lemma \ref{lem-1},
\begin{equation}\label{def_Omega}
\Omega=\left\{(x,y): 0\leq x\leq a, 0\leq y\leq\min\{\tfrac12(1+x^2), \sqrt{\tfrac13(1-x^2)}\}\right\}\ ,
\end{equation}
or, in other words, $\Omega = \Omega_1 \cup \Omega_2$ and
\[
\begin{split}
\Omega_1 &= \left\{(x,y) : 0\leq x \leq b,\, 0\leq y\leq \tfrac12(1+x^2) \right\},\\
\Omega_2 &= \left\{(x,y) : b\leq x \leq a,\, 0\leq y\leq \sqrt{\tfrac13(1-x^2)} \right\}\ .
\end{split}
\]
From now on, we define 
\[d:=\sqrt{\tfrac13(1-a^{2})}=0.386\ldots\ .\]

Now, let find maximal value of $f_2$ on $\Omega$.

\medskip

To derive the critical points of $f_2$, we need to solve in the interior of $\Omega$ the following system
\[
\left\{
\begin{split}
\frac{\partial f_2(x,y)}{\partial x} &= 12 x^2+6 y -\frac{2}{\sqrt5}\frac{x}{ \sqrt{1-x^2-3 y^2}} = 0,\\
\frac{\partial f_2(x,y)}{\partial y} &= 6 x- \frac{6}{\sqrt5} \frac{y}{\sqrt{1-x^2-3 y^2}} = 0.
\end{split}
\right.
\]
Now, from
\[3y\frac{\partial f_2(x,y)}{\partial x} - x \frac{\partial f_2(x,y)}{\partial y} =
18y^2+36x^2y-6x^2\ , \]
we receive
\begin{equation}\label{for_1}
x^2=\frac{3y^2}{1-6y}\ ,
\end{equation}
which is false for $y>1/6$.
 
Substituting it in the second equation of the system gives
\[15\left[(1-3y^2)(1-6y)-3y^2\right]-(1-6y)^2=0\ ,\]
which has only one solution in $[0,1/6]$; namely, $y'_1=0.153\ldots$. In this case, it follows from \eqref{for_1} that $x'_1=0.961\ldots$ which is greater than $a$. This means that the greatest value of $f_2$ is obtained on the boundary of $\Omega$.

We have:
\begin{itemize}
\item[-]  $f_2(0,y) = \frac{2}{\sqrt{5}} \sqrt{1-3 y^2}$ is decreasing for $y\in[0,1/2]$, so  $f_2(0,y)\le f_2(0,0) = \frac{2}{\sqrt{5}}  = 0.894\ldots$,  
\item[-]  for $f_2(a,y) = 4a^3+6ay+\frac{2}{\sqrt{5}}\sqrt{1-a^2-3 y^2}$, $y\in[0,d]$ we have $f_2(a,y) \le f_2(a,y_2) = 3.461\ldots$, where $y'_2=0.365\ldots$ is the unique real solution of the equation $\partial f_2(a,y)/\partial y=0$ on the interval $\left(0, d\right) $,
\item[-]  $f_2(x,0) = 4 x^3+\frac{2}{\sqrt{5}}\sqrt{1-x^2}$ is increasing on $(0,a)$, so $f_2(x,0) \le f_2(a,0) = 2.236\ldots$, 
\item[-]  $g_1(x) := f_2(x,(1+x^2)/2) = 3x+7 x^3+   \sqrt{1-10x^2-3x^4}/\sqrt{5}  \le g_1(x'_3) = 1.212\ldots$, 
where $x'_3=0.298\ldots$ is the unique real solution of the equation $g'_1(x)=0$ on the interval $\left(0, b\right)$,
\item[-]  $g_2(x) := f_2(x,\sqrt{1-x^2}/\sqrt{3}) = 4 x^3+2 \sqrt{3} x\sqrt{1-x^2}  \le g_2(a) = 3.360\ldots$, since $g_2$ is increasing on $(b,a)$.
\end{itemize}

Combining the analysis from (ii), we receive that the function $f_2$, on the domain $\Omega$ achieves the greatest value $3.461\ldots$ obtained for $x=a$ and $y=y'_2=0.365\ldots$, i.e.,
\[ |a_4| \le f_2(a,y_2) = 3.461\ldots.\]

\item[(iii)] We will work in a similar way as in (ii). Deriving $\omega_{33}$ and $ \omega_{35}$ from \eqref{eq-13} and applying them in the expression for $a_5$ from \eqref{eq-13} we obtain
\begin{equation}\label{for_a5}
a_{5}=2\omega_{17}+6\omega_{11}\omega_{15}+12\omega_{11}^{2}\omega_{13}+3\omega_{13}^{2}+5\omega_{11}^{4}\ ,
\end{equation}
and so
$$|a_{5}|\leq 2|\omega_{17}|+6|\omega_{11}||\omega_{15}|+12|\omega_{11}|^{2}|\omega_{13}|+3|\omega_{13}|^{2}+5|\omega_{11}|^{4}.$$
From \eqref{eq-8}, in a similar way as we received  \eqref{eq-11}, we can also obtain
$$|\omega_{11}|^2 + 3 |\omega_{13}|^2 + 5|\omega_{15}|^2 +7|\omega_{17}|^2\leq1,$$
which implies
\begin{eqnarray}\label{for_w17}
|\omega_{17}| &\leq& \frac{1}{\sqrt{7}}\sqrt{1-|\omega_{11}|^{2}-3|\omega_{13}|^{2}-5|\omega_{15}|^{2}} \notag \\
&\leq& \frac{1}{\sqrt{7}}\sqrt{1-|\omega_{11}|^{2}-3|\omega_{13}|^{2}}.
\end{eqnarray}
Finally, we get
\[
\begin{split}
|a_{5}| &\leq 5|\omega_{11}|^{4}+12|\omega_{11}|^{2}|\omega_{13}|+3|\omega_{13}|^{2}
\\
&+ \left(\frac{2}{\sqrt{7}}+\frac{6}{\sqrt{5}}|\omega_{11}|\right)\sqrt{1-|\omega_{11}|^{2}-3|\omega_{13}|^{2}}
 = :f_{3}(|\omega_{11}|,|\omega_{13}|),
\end{split}
\]
where
$$f_{3}(x,y) = 5x^{4} +12x^{2}y+3y^{2} + \left(\frac{2}{\sqrt{7}}+\frac{6}{\sqrt{5}}x\right)\sqrt{1-x^{2}-3y^{2}},$$
and $(x,y)\in \Omega$.

\medskip

For the critical points of $f_3$ in the interior of $\Omega$ we have
\[
\left\{
\begin{split}
\frac{\partial f_3(x,y)}{\partial x} &= 20 x^3 + 24xy-\frac{\left(\frac{6 x}{\sqrt{5}}+\frac{2}{\sqrt{7}}\right) x}{\sqrt{1-x^2-3 y^2}}\\
&\quad +\frac{6}{\sqrt{5}}\sqrt{1-x^2-3 y^2} = 0\\
\frac{\partial f_3(x,y)}{\partial y} &= 12 x^2+6 y -\frac{3 \left(\frac{6 x}{\sqrt{5}}+\frac{2}{\sqrt{7}}\right) y}{\sqrt{1-x^2-3 y^2}}= 0\ .
\end{split}
\right.
\]
Since
\begin{multline}\label{eq-q1}
 3y\frac{\partial f_3(x,y)}{\partial x} - x \frac{\partial f_3(x,y)}{\partial y} \\ =
\frac{18y}{\sqrt5} \sqrt{1-x^2-3 y^2}+12x^3(5y-1)+6xy(12y-1)\ ,
\end{multline}
we solve the system
\[
\left\{
\begin{split}
&\frac{18y}{\sqrt5} \sqrt{1-x^2-3 y^2}+12x^3(5y-1)+6xy(12y-1) = 0\\
&\sqrt{1-x^2-3 y^2} = \frac{\left(\frac{3 x}{\sqrt{5}}+\frac{1}{\sqrt{7}}\right) y}{2 x^2+y}\ .
\end{split}
\right.
\]

Using numerical computation we conclude that this system does not have solutions in the interior of $\Omega$. On the edges of $\Omega$ we apply classical calculus techniques and receive:
\begin{itemize}
  \item[-] $f_3(0,y) = 3 y^2+\frac{2}{\sqrt{7}}\sqrt{1-3 y^2}$ is increasing for $y\in[0,1/2]$, so  $f_3(0,y)\le f_3(0,1/2) = \frac{3}{4}+\frac{1}{\sqrt{7}} = 1.127\ldots$, 
  \item[-] $f_3(a,y) \le f_3(a,y'_4) = 4.993\ldots$ where $y'_4=0.338\ldots $ is the unique real solution of $\partial f_3(a,y)/\partial y=0$ on the interval $(0, d)$,
  \item[-] $f_3(x,0) = 5 x^4+\left(\frac{6 x}{\sqrt{5}}+\frac{2}{\sqrt{7}}\right) \sqrt{1-x^2} \le f_3(a,0) = 3.360\ldots$, since $f_3(x,0)$ is increasing on $(0,a)$,
  \item[-]  $g_3(x) := f_3(x,(1+x^2)/2) =\frac{1}{35}(21\sqrt5 x+5\sqrt7)\sqrt{1-10x^2-3x^4}+\frac14(3+30x^2+47x^4)
\le g_3(x'_5) = 1.748\ldots$, where $x'_5= 0.287\ldots$ is the unique real solution of the equation $g'_3(x)=0$ on $(0,b)$,
  \item[-]  $g_4(x) := f_3(x,\sqrt{1-x^2}/\sqrt{3}) =1-x^2+5x^4+4x^2\sqrt{3(1-x^2)} \le g_4(a) = 4.526\ldots$, since $g_4$ is strictly increasing on $(b,a)$.
\end{itemize}

Combining what we obtained in (iii), we receive the greatest value of $f_3$ on $\Omega$ being $4.993\ldots$, for $x=a$ and $y=y_2$, i.e.,
\[ |a_5| \le f_3(a,y'_4) = 4.993\ldots.\]
\end{itemize}
\end{proof}

%

Next we shall derive the upper bounds of the differences of the moduli of initial coefficients of a normalized bi-univalent function $f$ from $\mathcal{S}_{b}$, i.e.
\[ |a_4| - |a_3| \quad \mbox{and} \quad |a_5| - |a_4|.\]

Using \eqref{for_a4}, the second relation in \eqref{eq-13} and $\frac{|w_{11}|}{a}\leq 1$,  we obtain
\[
\begin{split}
|a_4| - |a_3| & \le |a_4| - \frac1a|\omega_{11}||a_3| \le \left|a_4 - \frac1a\omega_{11}a_3 \right|\\
&= \left| 2\omega_{15} + \left( 6-\frac2a \right) \omega_{11}\omega_{13} + \left( 4 -\frac3a \right) \omega_{11}^3 \right|\\
&\le 2|\omega_{15}| + \left( 6-\frac2a \right) |\omega_{11}||\omega_{13}| + \left( \frac3a -4 \right) |\omega_{11}|^3\\
&\le \left( \frac3a -4\right) |\omega_{11}|^3 + \left( 6-\frac2a \right) |\omega_{11}||\omega_{13}| + \frac{2}{\sqrt5}\sqrt{1-|\omega_{11}|^2-3|\omega_{13}|^2} \\
&= \left( \frac3a -4\right) x^3 + \left( 6-\frac2a \right) xy + \frac{2}{\sqrt5}\sqrt{1-x^2-3y^2} =: f_4(x,y),
\end{split}
 \]
where $x=|\omega_{11}|$, $y=|\omega_{13}|$ and $(x,y)\in\Omega$.

\medskip

In a similar way, using additionally \eqref{for_a4} and \eqref{for_w17}, we receive
\[
\begin{split}
|a_5| - |a_4| & \le |a_5| - \frac1a|\omega_{11}||a_4| \le \left|a_5 - \frac1a\omega_{11}a_4 \right|\\
& = \left| 2\omega_{17} + \left( 6-\frac2a \right) \omega_{11}\omega_{15} + \left( 12-\frac6a \right) \omega_{11}^2 \omega_{13} + 3 \omega_{13}^2 + \left( 5-\frac{4}{a} \right)\omega_{11}^4\right|\\
& \le 2|\omega_{17}| + \left( 6-\frac2a \right) |\omega_{11}||\omega_{15}| + \left( 12-\frac6a \right) |\omega_{11}|^2 |\omega_{13}|\\
&\quad + 3 |\omega_{13}|^2 + \left( \frac{4}{a} -5\right)|\omega_{11}|^4\\
&\le \left(\frac{4}{a}-5\right) |\omega_{11}|^4 + \left(12-\frac{6}{a}\right) |\omega_{11}|^2 |\omega_{13}|+3 |\omega_{13}|^2\\
&\quad + \left[ \frac{2}{\sqrt{7}} + \frac{1}{\sqrt{5}}\left(6-\frac{2}{a}\right) |\omega_{11}| \right]\sqrt{1-|\omega_{11}|^2-3 |\omega_{13}|^2}\\
&= \left(\frac{4}{a}-5\right) x^4 + \left(12-\frac{6}{a}\right) x^2 y+3 y^2\\
&\quad + \left[ \frac{2}{\sqrt{7}} + \frac{1}{\sqrt{5}}\left(6-\frac{2}{a}\right) x \right]\sqrt{1-x^2-3 y^2} =: f_5(x,y),
\end{split}
 \]
again with $x=|\omega_{11}|$, $y=|\omega_{13}|$ and $(x,y)\in\Omega$. Here and above, $\Omega$ is defined by \eqref{def_Omega}.

In order to obtain the estimates of the moduli of the above coefficient differences we continue with finding the greatest values of $f_4$ and $f_5$ on the domain $\Omega$ in a similar manner as in Theorem 1.

\begin{thm}\label{th2}
Let $f\in\mathcal{S}_{b}$ be given by \eqref{e1}. Then
\begin{itemize}
  \item[($i$)] $|a_4| - |a_{3}|\leq 1.174\ldots$,
  \item[($ii$)] $|a_5| - |a_{4}| \leq 1.822\ldots$.
\end{itemize}
\end{thm}

\medskip

\begin{proof}$ $
\begin{itemize}
\item[(i)] At the beginning we consider critical points of $f_4$ on $\Omega$. For this reason we look for the solutions of the following system
\[
\left\{
\begin{split}
\frac{\partial f_4(x,y)}{\partial x} &= \left(\frac{9}{a}-12\right) x^2+\left(6-\frac{2}{a}\right) y-\frac{2}{\sqrt{5}}\frac{x}{\sqrt{1-x^2-3 y^2}} = 0\\
\frac{\partial f_4(x,y)}{\partial y} &= \left(6-\frac{2}{a}\right) x-\frac{6}{\sqrt{5}}\frac{y}{\sqrt{1-x^2-3 y^2}} = 0
\end{split}
\right..
\]
Hence,
\begin{multline*}
 3y\frac{\partial f_4(x,y)}{\partial x} - x \frac{\partial f_4(x,y)}{\partial y}\\
 = \frac1a\left[ -x^2 [9y(4a-3)+6a-2]+6 (3 a-1) y^2 \right] = 0,
 \end{multline*}
so
\[ x = \frac{y\sqrt{6} \sqrt{3 a -1}}{\sqrt{9y(4a-3)+6a-2}} := h_1(y)\]
Putting it into the second equation of the system we receive an equation of variable $y$. Numerical calculation shows that this equation has only one real solution in $(0,1)$, that is $y'_6=0.358\ldots$, such that $(h_1(y'_6),y'_6)=(0.634\ldots,0.358\ldots)\in\Omega$ and $f_4(h_1(y'_6),y'_6)=1.174\ldots$.

Finally, we need to find the greatest value of $f_4$ on the edges of $\Omega$:
\begin{itemize}
\item[-]  $f_4(0,y) = \frac{2}{\sqrt{5}}\sqrt{1-3 y^2} \le f_4(0,0) = \frac{2}{\sqrt5} = 0.894\ldots$,
\item[-]  $f_4(a,y) \le f_4(a,y'_7) = 1.139\ldots$, where $y'_7=0.327\ldots$ is the unique real solution of the equation $f_4'(a,y)=0$ on the interval $(0, d)$,
\item[-]  $f_4(x,0) =\left( \frac3a -4\right) x^3  + \frac{2}{\sqrt5}\sqrt{1-x^2} \le f_4(0,0) = \frac{2}{\sqrt5} = 0.894\ldots$, since $f_4(x,0)$ is a decreasing function on $(0,a)$,
\item[-]  $g_5(x) := f_4(x,(1+x^2)/2)=\tfrac15\sqrt{5-50x^2-15x^4}-x\left[\left(1-\frac{2}{a}\right)x^2-3+\frac{1}{a}\right]$, so  $g_5(x) \le g_5(x'_8) = 0.709\ldots$, where $x'_8=0.252\ldots$ is the unique real solution of $g'_5(x)=0$ on the interval $(0,b)$,
\item[-]  $g_6(x) := f_4(x,\sqrt{1-x^2}/\sqrt{3}) =2\left(1-\frac{1}{3a}\right)\sqrt{3(1-x^2)}+\left(\frac{3}{a}-4\right)x^3$, so  $g_6(x)\le g_6(x'_9) = 0.969\ldots$, where $x'_9=0.715\ldots$ is the unique real solution of $g'_3(x)=0$ on the interval $(b,a)$.
\end{itemize}

So, on the domain $\Omega$, the function $f_4$ achieves the greatest value $1.174\ldots$. Consequently,
\[ |a_4| - |a_3|\le 1.174\ldots.\]

\item[(ii)] Using Wolfram Mathematica we receive that the system of equations \linebreak $\frac{\partial f_5(x,y)}{\partial x} = 0$ and  $\frac{\partial f_5(x,y)}{\partial y} = 0$, has only one solution in the interior of $\Omega$, that is $(0.717\ldots, 0.312\ldots)$ with value $1.822\ldots$. On the edges of $\Omega$ we have:
\begin{itemize}
\item[-]  $f_5(0,y) =3 y^2 + \frac{2}{\sqrt{7}}\sqrt{1-3 y^2}\le f_5(0,1/2) = 1.127\ldots$ since $f_5(0,y)$ is increasing on $(0,1/2)$,
\item[-]  $f_5(a,y) \le f_5(a,y'_{10}) = 1.819\ldots$, where $y'_{10}=0.300\ldots$ is the unique real solution of the equation $f_5'(a,y)=0$ on the interval $(0, \d)$,
\item[-]  $f_5(x,0)=\left(\frac{4}{a}-5\right) x^4  + \left[ \frac{2}{\sqrt{7}} + \frac{1}{\sqrt{5}}\left(6-\frac{2}{a}\right) x \right]\sqrt{1-x^2} \le f_5(a,y'_{11}) = 1.374\ldots$, where $y'_{11}=0.667\ldots$ is the unique real solution of $f_5'(x,0)=0$ on $(0,a)$,
\item[-]  $g_7(x) := f_5(x,(1+x^2)/2)= \left[ \frac{1}{\sqrt{7}} + \frac{1}{\sqrt{5}}\left(3-\frac{1}{a}\right) x \right]\sqrt{1-10x^2-3 x^4} + \frac14(3+30x^2+7x^4)+\frac{1}{a}(x^4-3x^2)$, so $g_7(x) \le g_7(x'_{12}) = 1.317\ldots$, where $x'_{12}=0.247\ldots$ is the unique real solution of $g'_7(x)=0$ on $(0,b)$,
\item[-]  $g_8(x) := f_5(x,\sqrt{1-x^2}/\sqrt{3}) = 2\left(2-\frac{1}{a}\right)x^2\sqrt{3(1-x^2)}+1-x^2+\left(\frac{4}{a}-5\right)x^4$ and $g_8(x)\le g_8(a) = 1.402 \ldots$, since $g_8$ is strictly increasing on $(b,a)$.
\end{itemize}

So, on the domain $\Omega$, the function $f_5$ takes the greatest value $1.822\ldots$ and for this reason
\[ |a_5| - |a_4|\le 1.822\ldots.\]
\end{itemize}
\end{proof}

\section{The second Hankel determinant}

The Hankel determinant $H_{q,n}(f)$ of a given function $f$, for $q\geq 1$ and $n\geq 1$, is defined by
\[
        H_{q,n}(f) = \left |
        \begin{array}{cccc}
        a_{n} & a_{n+1}& \ldots& a_{n+q-1}\\
        a_{n+1}&a_{n+2}& \ldots& a_{n+q}\\
        \vdots&\vdots&~&\vdots \\
        a_{n+q-1}& a_{n+q}&\ldots&a_{n+2q-2}\\
        \end{array}
        \right |.
\]
Hankel determinants find application in the theory of singularities \cite{dienes}, as well as in the study of power series with integer coefficients (\cite{cantor,edrei,polya}).

\medskip

In the past several years much attention is given on finding the bounds (preferably sharp) of the moduli of Hankel determinants for various classes of analytic functions. Since the study of the general case faces serious  calculation challenges, the determinants of second and the third order, defined respectively by
\[H_{2,2}(f)= \left|\begin{array}{cc}
        a_2& a_3\\
        a_3& a_4
        \end{array}
        \right | = a_2a_4-a_{3}^2\] and
\[ H_{3,1}(f) =  \left |
        \begin{array}{ccc}
        1 & a_2& a_3\\
        a_2 & a_3& a_4\\
        a_3 & a_4& a_5\\
        \end{array}
        \right | = a_3(a_2a_4-a_{3}^2)-a_4(a_4-a_2a_3)+a_5(a_3-a_2^2),
\]
are studied instead.

\begin{thm}\label{th3}
Let $f\in\mathcal{S}_{b}$ be given by \eqref{e1}. Then
\[|H_{2,2}(f)| \leq 1.280\ldots.\]
\end{thm}

\medskip

\begin{proof}
Using the expressions from \eqref{eq-13} we receive
\[
\begin{split}
|a_2a_4-a_3^3| &= |-\omega_{11}^4 -4 \omega_{13}^2 +4 \omega_{11} \omega_{15}|\\
&\le |\omega_{11}|^4 + 4 |\omega_{13}|^2 + 4 |\omega_{11}| \frac{1}{\sqrt5}\sqrt{1-|\omega_{11}|^2-3|\omega_{13}|^2}\\
&= x^4 + 4 y^2 + 4 x\frac{1}{\sqrt5}\sqrt{1-x^2-3y^2} := f_6(x,y),
\end{split}
\]
where $x=|\omega_{11}|$, $y=|\omega_{13}|$, $(x,y)\in\Omega$, with$\Omega$ defined by \eqref{def_Omega}.

\medskip

We start with deriving critical points of $f_6$ inside $\Omega$ from the system
\[
\left\{
\begin{split}
\frac{\partial f_6(x,y)}{\partial x} &= 4 x^3+\frac{4 (1-2x^2-3y^2)}{\sqrt{5} \sqrt{1-x^2-3 y^2}} = 0\\
\frac{\partial f_6(x,y)}{\partial y} &= 4y \left(2-\frac{3 x}{\sqrt{5} \sqrt{1-x^2-3 y^2}}\right) = 0
\end{split}
\right..
\]
From the second equation $y=\frac{\sqrt{20-29 x^2}}{2 \sqrt{15}} =: h_2(x)$, which applied in the first one, after simplification, gives
\[ x\left(x^2-\frac{11}{30}\right) = 0.\]
This leads to a unique solution $(x'_{13},y'_{13})$ of the system in the interior of $\Omega$, where $x'_{13}=\sqrt{\frac{11}{30}}=0.605\ldots$ and $y'_{13}=h_2(x'_{13})=\frac{1}{30}\sqrt{\frac{281}{2}}=0.395\ldots$. Moreover, $f_6(x'_{13},y'_{13})= 1079/900 = 1.198\ldots$. 

On the edges of $\Omega$ we have:
\begin{itemize}
\item[-]  $f_6(0,y) = 4 y^2 \le f_6(0,1/\sqrt3) = 4/3 = 1.(3)$,
\item[-]  $f_6(a,y) \le f_6(a,y'_{14}) = 1.232\ldots$, where $y'_{14}=0.258\ldots$ is the unique real solution of the equation $f_6'(a,y)=0$ on the interval $(0, d)$,
\item[-]  $f_6(x,0) \le f_6(a,0) = 1.193\ldots$, since $f_6(x,0)$ is an increasing function on $(0,a)$,
\item[-]  $g_{9}(x) := f_6(x,(1+x^2)/2) \leq g_{9}(x'_{15})=1.280\ldots$, where $x'_{15}=0.281\ldots$ is the unique real solution of the equation $g'_{9}(x)=0$ on the interval $(0, b)$;
\item[-]  $g_{10}(x) := f_6(x,\sqrt{1-x^2}/\sqrt{3}) =x^4+\tfrac43(1-x^2)$, $x\in[b,a]$, so $g_{10}(x)\leq \max\{g_{10}(b), g_{10}(a)\}=g_{10}(b)=1.213\ldots$.
\end{itemize}

\medskip

Summing up,  the greatest value of $f_6$ in $\Omega$ is equal to $1.280\ldots$.
\end{proof}

\begin{rem}
In \cite{2026-zapeawa} was proven that for $f$ in the class $\mathcal{S}$, $|H_{3,1}(f)|\leq \frac23\sqrt{\frac{73}{35}}=0.962\ldots$. Since $\mathcal{S}_b\subset \mathcal{S}$, so the same inequality holds true also for the class $\mathcal{S}_b$.
\end{rem}

\medskip

\section{The third and the forth logarithmic coefficients}

The logarithmic coefficient, $\gamma_n$, of a univalent function $f$  are defined by
  \be\label{log-co-1}
  F_f(z) := \log\frac{f(z)}{z}=2\sum_{n=1}^\infty \gamma_n z^n.
  \ee

\medskip

Relatively little exact information is known about the coefficients $\gamma_n$. For the class $\mathcal{S}$ a natural conjecture $|\gamma_n|\le1/n$, inspired by the Koebe function (whose logarithmic coefficients are $1/n$) is false, even in order of magnitude (see Duren \cite[Section 8.1]{duren}).
The sharp estimates of logarithmic coefficients of $f\in\mathcal{S}$ are known only for $\gamma_1$ and $\gamma_2$, namely,
\[|\gamma_1|\le1\quad\mbox{and}\quad |\gamma_2|\le \frac12+\frac1e=0.635\ldots .\]

Recently, for the functions $f$ from the general class of univalent functions was proven that $|\gamma_{3}|\leq 0.556\ldots$ (\cite{124}), $|\gamma_4| \le 0.510\ldots$ (\cite{106}), and $|\gamma _{3}|-|\gamma_{2}|\leq \frac{1}{\sqrt{5}}$, $|\gamma _{4}|-|\gamma_{3}|\leq \frac{1}{\sqrt{7}}$ (\cite{113}). In this section we discuss the estimates of the moduli of the third and the fourth logarithmic coefficients for functions in $\mathcal{S}_{b}$.

From the relations \eqref{e1} and \eqref{log-co-1}, after  equating the coefficients we receive:
\[
\begin{split}
 \gamma_{1}&=\frac{a_{2}}{2},\\
 \gamma_{2}&=\frac{1}{2}\left(a_3-\frac{1}{2}a_{2}^{2}\right), \\
 \gamma_3 &=\frac{1}{2}\left(a_4-a_2a_3+\frac13a_2^3\right),\\
\gamma_4 &=\frac{1}{2}\left(a_5-a_2a_4-\frac{1}{2}a_3^2+a_{2}^{2}a_{3}-\frac{1}{4}a_{2}^{4}\right).
\end{split}
\]

So, for $f$ in $\mathcal{S}_{b}$, $|\gamma_1| \le a/2 = 0.7425$.

\medskip

Further, for the next three logarithmic coefficients, in a similar way as before, using \eqref{eq-13}, we receive
\[
\begin{split}
 |\gamma_{2}| &= \frac{1}{2} \left(x^2+2 y\right) := f_7(x,y), \\
 |\gamma_3| &= \frac{x^3}{3} + xy + \frac{1}{\sqrt{5}}\sqrt{1-x^2-3 y^2} := f_8(x,y),\\
 |\gamma_4| &= \frac{x^4}{4} +x^2 y+\frac{y^2}{2} + \left(\frac{x}{\sqrt{5}}+\frac{1}{\sqrt{7}}\right)\sqrt{1-x^2-3 y^2} := f_9(x,y),
\end{split}
\]
with $x=|\omega_{11}|$, $y=|\omega_{13}|$, $(x,y)\in\Omega$ defined by\eqref{def_Omega}.

By studying the functions $f_7$, $f_8$, and $f_9$ as functions of two variables we can obtain estimates of $\gamma_k$, $k=2,3,4$.

\medskip
For the second and the fourth logarithmic coefficient, in a similar way as in the previous theorems, we receive estimates $|\gamma_2|\le 0.662\ldots$ and $|\gamma_{4}| \leq 0.613\ldots$, which are weaker then the estimates for the general class $\mathcal{S}$, namely:  $0.635\ldots$ and $0.510\ldots$, respectively.

For the third logarithmi0c coefficient, applying the similar technique as in the previous theorems, we receive that $f_8$ has no critical points in the interior of $\Omega$, while on the edges achieves the greatest value $0.551\ldots$ for $x=a$ and $y'_{16}=0.267\ldots$.

\medskip

\begin{thm}\label{th4}
Let $f\in\mathcal{S}_{b}$ be given by \eqref{e1}. Then
\[|\gamma_{3}| \leq 0.551\ldots.\]
\end{thm}

\end{document}